\documentclass{amsart}
\usepackage{amssymb}
\usepackage{color}

\newtheorem{theorem}{Theorem}[section]

\newtheorem{lemma}[theorem]{Lemma}

\newcommand{\A}{\mathcal A}

\newcommand{\F}{\mathcal{F}}

\newcommand{\K}{\mathbb K}

\renewcommand{\rho}{\varrho}

\newcommand{\w}{\omega}
\newcommand{\IU}{\mathbb U}
\newcommand{\IN}{\mathbb N}
\newcommand{\IK}{\mathbb K}

\title{Embedding topological fractals in universal spaces}
\author{Taras Banakh and Filip Strobin}
\address{Jan Kochanowski University in Kielce (Poland), and Ivan Franko National University of Lviv (Ukraine)}
\email{t.o.banakh@gmail.com}
\address{Institute of Mathematics, Jan Kochanowski University in Kielce, (Poland), and Institute of Mathematics, \L\'od\'z University of Technology, W\'olcza\'nska 215, 93-005, \L\'od\'z (Poland)}
\email {filip.strobin@p.lodz.pl}
\subjclass[2010]{Primary: 28A80; Secondary: 37C25, 37C70}
\keywords{Topological fractal, universal Urysohn space, Banach contraction, Rakotch contraction}
\date{}

\begin{document}

\begin{abstract} {Let $X$ be a universal (Urysohn) space. We prove that every topological fractal is homeomorphic (isometric) to the attractor $A_\F$ of a function system $\F$ on $X$ consisting of Rakotch contractions.}
\end{abstract}
\maketitle

\section{Introduction}
Let $X$ be a topological space. By a {\em function system} on $X$ we shall understand any finite family $\F$ of continuous self-maps of $X$. Every function system $\F$ generates the mapping
$$\F:\K(X)\to \K(X),\;\;\F:K\mapsto \bigcup_{f\in\F}f(K),$$on the hyperspace $\K(X)$ of non-empty compact subsets of $X$, endowed with the Vietoris topology. If the topology of $X$ is generated by a metric $d$, then the Vietoris topology on $\K(X)$ is generated by the Hausdorff metric
$$d_H(A,B)=\max\{\max_{a\in A}d(a,B),\max_{b\in B}d(A,b)\}.$$
We shall say that a compact set $A\in\K(X)$ is an {\em attractor} of a function system $\F$ if $\F(A)=A$ and for every compact set $K\in\K(X)$ the sequences of iterations $\F^n(K)=\F\circ\dots\circ\F(K)$ converges to $A$ in the hyperspace $\K(X)$. A function system $\F$ on a Hausdorff topological space $X$ can have at most one attractor, which will be denoted by $A_\F$. The following classical result of a Hutchinson-Barnsley theory of fractals \cite{H} detects function systems possessing attractors.

\begin{theorem}\label{T1} Each function system $\F$ consisting of Banach contractions of a complete metric space $X$ has a unique attractor $A_\F$.
\end{theorem}

The proof of Theorem \ref{T1} given in \cite{H} (cf. also \cite{B}) uses the observation that a Banach contracting function system $\F$ {(i.e., a function system consisting of Banach contractions)} on a complete metric space $X$ induces a Banach contracting map $\F:\K(X)\to\K(X)$ of the hyperspace, which makes possible to apply {the} Banach Contracting Principle to show that $\F$ has an attractor $A_\F$.

It turns out that the Banach contractivity of $\F$ in Theorem~\ref{T1} can be weakened to the $\varphi$-contractivity, which is defined as follows.

A map $f:X\to Y$ between metric spaces $(X,d_X)$ and $(Y,d_Y)$ is called  {\em $\varphi$-contracting} for a function $\varphi:[0,\infty)\to[0,\infty)$ if $d_Y(f(x),f(x'))\le \varphi(d_X(x,x'))$ for every points $x,x'\in X$. It follows that $f:X\to Y$ is Banach contracting {(i.e., it is a Lipschitz mapping with the Lipschitz constant less than $1$)} if and only if it is $\varphi$-contracting for some function $\varphi:[0,\infty)\to[0,\infty)$ such that $\sup_{0<t<\infty}\varphi(t)/t<1$.

A function $f:X\to Y$ is called
\begin{itemize}
\item {\em Rakotch contracting} if $f$ is $\varphi$-contracting for a function $\varphi:[0,\infty)\to[0,\infty)$ such that $\sup_{a<t<\infty}\varphi(t)/t<1$ for every $a>0$;
\item {\em Matkowski contracting} if $f$ is $\varphi$-contracting for a function $\varphi:[0,\infty)\to[0,\infty)$ such that $\lim_{n\to\infty}\varphi^n(t)=0$ for every $t>0$, where $\varphi^n$ denotes the $n$th iteration of $\varphi$;
\item {\em Edelstein contracting} if $d_Y(f(x),f(x'))<d_X(x,x')$ for any distinct points $x,x'\in X$.
\end{itemize}
It is known (\cite{JJ}) that {Rakotch contracting maps are Matkowski contracting, Matkowski contracting are Edelstein contracting} and that if $X$ is compact, then $f:X\to Y$ is Rakotch contracting if and only if it is Edelstein contracting. {The notions of Rakotch, Matkowski and Edelstein contracting maps are connected with certain generalizations of the Banach Contracting Principle, cf. \cite{E}, \cite{Mat}, \cite{Ra}.}

A topological version of Theorem~\ref{T1} was recently proved by Mihail \cite{M} who introduced the following notion (cf. also \cite{BI} for a particular vesion of it): A function system $\F$ on a Hausdorff topological space $X$ is called  \emph{topologically contracting} if
\begin{itemize}
\item for every $K\in\K(X)$, there is $D\in\K(X)$ such that $K\subset D$ and $\F(D)\subset D$;
\item for every $D\in\K(X)$ with $\F(D)\subset D$, and a sequence $\vec f=(f_n)_{n\in\w}\in\F^\w$, the set
$\bigcap_{n\in\w}f_0\circ\dots\circ f_n(D)
$
is a singleton.
\end{itemize}
It can be seen that in this case the singleton $\{\pi(\vec f)\}=\bigcap_{n\in\w}f_0\circ\dots\circ f_n(D)$ does not depend on the choice of the compact set $D$ and the map $\pi:\F^\w\to X$, $\pi:\vec f\mapsto \pi(\vec f)$, is continuous {(here  $\F^\w$ carries the topology of Tychonoff product of countably many copies of the finite space $\F$ endowed with the discrete topology)}. Moreover, the compact metrizable space $A_\F=\pi(\F^\w)$ is the attractor of the function system $\F$. This fact was proved by Mihail \cite{M}:

\begin{theorem}\label{T2} Every topologically contracting function system $\F$ on a Hausdorff topological space $X$ has an attractor $A_\F$, which is a compact metrizable space.
\end{theorem}

A Hausdorff topological space $X$ is called a {\em topological fractal} if $X=\bigcup_{f\in\F}f(X)$ for some topologically contracting function system $\F$ on $X$. It follows that for every topologically contractive function system $\F$ on a Hausdorff topological space $X$ its attractor $A_\F$ is a topological fractal.
Mihail's Theorem~\ref{T2} implies that each topological fractal is a compact metrizable space. Moreover, the topology of $X$ is generated by a metric $d$ making all maps $f\in\F$ Rakotch contracting (see \cite{BKNNS} or \cite{MM}).
Topological fractals were introduced and investigated by Kameyama \cite{K} (who called them \emph{self similar sets}) and considered also in \cite{BN} and \cite{D}.

In this paper we shall search for copies of (Banach) topological fractals in universal (metric) spaces. A topological space $X$ will be called {\em topologically universal} if every compact metrizable space $K$ admits a topological embedding into $X$. The following realization theorem will be proved in Section~\ref{s2}.

\begin{theorem}\label{main1} If a Tychonoff space $X$ is topologically universal, then every topological fractal is homeomorphic to the attractor $A_\F$ of a topologically contractive function system $\F$ on $X$. If the space $X$ is metrizable, then we can additionally assume that all maps $f\in\F$ are Rakotch contracting with respect to some bounded metric $d$  generating the topology on $X$.
\end{theorem}

For the universal Urysohn space we can prove a bit more. Let us recall that {the} {\em universal Urysohn space} is a separable complete metric space $\IU$ such that each isometric embedding $f:B\to \IU$ of a subspace $B$ of a finite metric space $A$ extends to an isometric embeddig $\bar f:A\to\IU$. By \cite{Urysohn},  a universal Urysohn space exists and is unique up to a bijective isometry.

A compact metric space $X$ will be called a ({\em Banach}) {\em Rakotch fractal} if $X=\bigcup_{f\in\F}f(X)$ for some function system $\F$ consisting of (Banach) Rakotch contractions of $X$. By \cite{BKNNS} and \cite{MM}, each topological fractal is homeomorphic to a Rakotch fractal. On the other hand, there are examples of Rakotch fractals which are not homeomorphic to Banach fractals (see \cite{BN}, \cite{K}, \cite{NS}). {The following realization theorem will be proved in Section~\ref{s3}.}

\begin{theorem}\label{main2} Each (Banach) Rakotch fractal $X$ is isometric to the attractor $A_\F$ of a function system $\F$ consisting of Banach (Rakotch) contractions of the universal Urysohn space $\IU$.
\end{theorem}

\section{Copies of topological fractals in universal spaces}\label{s2}

In this section we shall prove Theorem~\ref{main1}.
At first we need to recall some information on spaces of probability measures.

For a compact metric space $(X,d_X)$ by $PX$ we shall denote the space of Borel probability measures on $X$ endowed with the metric
$$d_{PX}(\mu,\eta):=\inf\Big\{\underset{X\times X}{\int}d_X(x,y)\;d\lambda:\lambda\in \mathcal{B}(\mu,\eta)\Big\},
$$
where $\mathcal{B}(\mu,\eta)$ is the space of all Borel probability measures on $X\times X$ such that $\pi_1(\lambda)=\mu$ and $\pi_2(\lambda)=\eta$ {(here $\pi_1$ and $\pi_2$ stand for the projections onto the first and the second coordinate, respectively)}. It is known that $(PX,d_{PX})$ is a compact metric space and for every measures $\mu,\eta\in PX$, there is $\lambda\in \mathcal{B}(\mu,\eta)$ such that $d_{PX}(\mu,\eta)=\int_{X\times X}d(x,y)\;d\lambda$ (cf. \cite[Chapter 8, especially Theorem 8.9.3, Theorem 8.10.45, p. 234]{Bo}; here we use the fact that every Borel probability measure on a compact metrizable spaces is Radon \cite[p. 70]{Bo}).

{It follows that the mapping $X\ni x\to \delta_x\in PX$ assigning to each point $x\in X$ the Dirac measure $\delta_x$ supported at $x$ is an isometric embedding of $X$ into $PX$.}
Every continuous map $f:X\to Y$ between compact metric spaces $(X,d_X)$ and $(Y,d_Y)$ induces a continuous map $Pf:PX\to PY$ between their spaces of probability measures. The map $Pf$ assigns to each measure $\mu\in PX$ the measure $Pf(\mu)\in PY$ defined by $Pf(\mu)(B)=\mu(f^{-1}(B))$ for a Borel subset $B\subset Y$.

\begin{lemma}\label{l2.1} If a map $f:X\to Y$ between compact metric spaces $X,Y$ is Rakotch contracting, then the induced map $Pf:P(X)\to P(Y)$ is Rakotch contracting too.
\end{lemma}

\begin{proof} Being Rakotch contracting, the map $f$ is $\varphi$-contracting for some function $\varphi:[0,\infty)\to[0,\infty)$ such that $c_\delta=\sup_{\delta\le t<\infty}\varphi(t)/t<1$ for every $\delta>0$.
By the compactness of $PX$, the Rakotch contractivity of $Pf$ is equivalent to its Edelstein contractivity. So, it suffices to prove that $d_{PY}(Pf(\mu),Pf(\eta))<d_{PX}(\mu,\eta)$ for any distinct measures $\mu,\eta\in PX$.
{So take any distinct $\mu,\eta\in PX$. As stated earlier}, there is a measure $\lambda\in\mathcal B(\mu,\eta)$ such that $d_{PX}(\mu,\eta)=\int_{X\times X}d_X(x,y) d\lambda$.
Since $d_{PX}(\mu,\eta)>0$, for some $\delta>0$ the compact set $X_\delta=\{(x,x')\in X\times X:d_X(x,x')\ge\delta\}$ has positive measure $\lambda(X_\delta)>0$.

Let $\tilde \lambda\in P(Y\times Y)$ be the image of the measure $\lambda$ under the map $f{\times} f:X\times X\to Y\times Y$, $f\times f:(x,y)\mapsto (f(x),f(y))$. Observe that for every Borel subset $B\subset Y$ we get
$$
\tilde \lambda(B\times Y)=\lambda((f{\times} f)^{-1}(B\times Y))=\lambda(f^{-1}(B)\times f^{-1}(Y))=\lambda(f^{-1}(B)\times X)=\mu(f^{-1}(B))=Pf(\mu)(B)
$$
and similarly
$$
\tilde \lambda(Y\times B)=Pf(\eta)(B),
$$which means that $\tilde \lambda\in\mathcal B(Pf(\mu),Pf(\eta))$. Moreover,
$$
d_{PY}(Pf(\mu),Pf(\eta))\leq\int_{Y\times Y}d_Y(y,y')\;d\tilde \lambda=\int_{X\times X}d_X(f(x),f(x'))\;d\lambda<\int_{X\times X}d_X(x,x')\;d\lambda=d_{PX}(\mu,\eta).
$$
The last strict inequality follows from $\lambda(X_\delta)>0$ and the fact that $d_Y(f(x),f(x'))\le c_\delta \cdot d_X(x,x')<d_X(x,x')$ for every $(x,x')\in X_\delta$.
\end{proof}

We shall also need a metrization theorem for globally contracting function systems, proved in \cite{BKNNS}. A function system $\F$ on a Hausdorff topological space $X$ is called {\em globally contracting} {(\cite[Definition 2.1]{BKNNS})} if there exists a non-empty compact set $K\subset X$ such that $\F(K)\subset K$ and for every open cover $\mathcal U$ of $X$ there is $n\in\mathbb N$ such that for every map $f\in\F^n=\{f_1\circ\dots\circ f_n:f_1,\dots,f_n\in\F\}$ the set $f(X)$ is contained in some set $U\in\mathcal U$. {The following result was proved in \cite[Theorem 6.7]{BKNNS}.}

\begin{theorem}\label{global} A function system $\F$ on a metrizable space $X$ is globally contractive if and only if the topology of $X$ is generated by a bounded metric $d$ making all maps $f\in\F$ Rakotch contractive.
\end{theorem}

Now we are able to present:
\smallskip

\noindent {\em Proof of Theorem~\ref{main1}.} Assuming that $K$ is a topological fractal, find a topologically contracting function system $\F$ on $K$ such that $K=\bigcup_{f\in\F}f(K)$. By \cite[Theorem 6.8]{BKNNS}, the topology of $K$ is generated by a metric $d_K$ making all maps $f\in\F$ Rakotch contracting. By Lemma~\ref{l2.1}, the function system $P\F=\{Pf:f\in\F\}$ consists of Rakotch contracting self-maps of the metric space $(PK,d_{PK})$.

Given any topologically universal Tychonoff space $X$, identify the compact metrizable space $PK$ with a (closed) subspace of $X$. The space $PK$, being a metrizable compact convex subset of a locally convex space, is an absolute retract in the class of Tychonoff spaces (this follows from \cite{Dug} and Tietze-Urysohn Theorem \cite[2.1.8]{En}). This implies that each map $Pf:PK\to PK$, $f\in\F$, can be extended to a continuous map $\overline{Pf}:X\to PK$. The Rakotch contractivity of $P\F$ {and Theorem \ref{global}} implies that the function system $P\F$ is globally contractive and so is the function system {$\overline{P\F}=\{\overline{Pf}:Pf\in P\F\}$} (as {$\overline{P\F}(X)\subset PK$)}.

The global contractivity of $\overline{P\F}$ implies the topological contractivity of $\overline{P\F}$ {(\cite[Theorem 2.2]{BKNNS})}. Then the function system $\overline{P\F}$ has a unique attractor, which coincides with $K$ by the uniquenes of the fixed point of the map $\overline{P\F}:\K(X)\to\K(X)$. Therefore, the topological fractal $K$ is homeomorphic to the attractor of the topologically contracting function system $\overline{P\F}$ on $X$.

If the space $X$ is metrizable, then by Theorem~\ref{global}, the topology of $X$ is generated by a bounded metric $d$ making all maps $f\in\overline{P\F}$ Rakotch contracting.\hfill$\square$
\smallskip

\section{Embedding fractals into the Urysohn universal space}\label{s3}

Recall that the Urysohn space $\IU$ is the unique (up to isometry) complete separable metric space $\IU$ such that every isometric embedding $f:B\to \IU$ {of a finite subset $B\subset \IU$ extends to an isometric embedding $\bar f:\IU\to\IU$, and any separable metric space is isometric to a subspace of $\IU$}. According to {\cite[Theorem 4.1]{Mel}}, the universal Urysohn space has a stronger universality property: every isometric embedding $f:B\to \IU$ of a {compact subspace $B\subset \IU$ extends to an isometric embedding $\bar f:\IU\to\IU$}. In fact, isometric embeddings in this result can be replaced by maps with given oscillation.

For a map $f:X\to Y$ between two metric spaces $(X,d_X)$ and $(Y,d_Y)$ its {\em oscillation} is the function $\w_f:[0,\infty]\to[0,\infty]$ assigning to each $\delta\in[0,\infty]$ the number
$$\w_f(\delta)=\sup\{d_Y(f(x),f(x')):x,x'\in X,\;d_X(x,x')\le\delta\}\in[0,\infty].$$
It follows that $d_Y(f(x),f(x'))\le\w_f(d_X(x,x'))$ for every points $x,x'\in X$. It is clear that the map $f:X\to Y$ is uniformly continuous if and only if $\lim_{\delta\to 0}\w_f(\delta)=0$.

If the metric space $(X,d_X)$ is {\em geodesic} (in the sense that for every points $x,x'\in X$ there is an isometric embedding $\gamma:[0,d_X(x,x')]\to X$ such that $\gamma(0)=x$ and $\gamma(d_X(x,x'))=x'$), then for any map $f:X\to Y$ its oscillation $\w_f$ is {\em subadditive} in the sense that $\w_f(s+t)\le\w_f(s)+\w_f(t)$ for any $s,t\in [0,\infty)$.
This motivates the following definition.

By a {\em continuity modulus} we shall understand any continuous subadditive function  $\varphi:[0,\infty)\to[0,\infty)$ with $\varphi(0)=0$. {It is easy to see that each continuous, concave function $\varphi:[0,\infty)\to[0,\infty)$ with $\varphi(0)=0$ is a continuity modulus. The following lemma uses some ideas from \cite{KR}.}

\begin{lemma}\label{l3.1} Let $\varphi$ be a continuity modulus. Every (injective) map
$f:B\to \IU$ with $\w_f\le\varphi$ defined on a compact subset $B$ of a separable metric space $(A,d_A)$ extends to an (injective) map $\bar f:A\to\IU$ with continuity modulus $\w_{\bar f}\le\varphi$.
\end{lemma}

\begin{proof} Fix a countable dense subset $\{a_n\}_{n\in\IN}$ in $A$ and put $B_0=B$ and $B_{n+1}=B_n\cup\{a_{n+1}\}$ for $n\in\w$. Let $f_0=f$. By induction we shall construct a sequence of (injective) maps $(f_n:B_n\to \IU)_{n\in\w}$ such that {$f_{n+1}|B_{n}=f_n$ and $\w_{f_{n}}\le\varphi$} for all $n\in\w$. Assume that for some $n\in\w$ the map $f_n:B_n\to\IU$ has been constructed. Consider the compact space $B_{n+1}=B_n\cup \{a_{n+1}\}$. If $B_{n+1}=B_n$, then put $f_{n+1}=f_n$. If $B_{n+1}\ne B_n$, then $a_{n+1}\notin B_n$. Consider the subspace $f_n(B_n)\subset \IU$ of the Urysohn space and fix any point $y\notin f_n(B_n)$. On the union $Y=f_n(B_n)\cup\{y\}$ consider the metric $d_Y$ which coincides on $f_n(B_n)$ with the metric $d_{\IU}$ of the Urysohn space $\IU$ and
$$d_Y(z,y)=\min\{d_{\IU}(z,f_n(b))+\varphi(d_A(b,a_{n+1})):b\in B_n\}$$ for $z\in f_n(B_n)$. It follows from the compactness of $B_n$ and $a_{n+1}\notin B_n$ that $d_Y(z,y)>0$ for every $z\in f_n(B_n)$. Let us show that the metric $d_Y$ satisfies the triangle inequality and hence is well-defined. Indeed, for any points $z,z'\in f_n(B_n)$, we can find points $b,b'\in B_n$ such that $d_Y(z,y)=d_\IU(z,f_n(b))+\varphi(d_A(b,a_{n+1}))$ and $d_Y(z',y)=d_\IU(z',f_n(b'))+\varphi(d_A(b',a_{n+1}))$. Then
$$
\begin{aligned}
d_Y(z,z')&=d_{\IU}(z,z')\le d_\IU(z,f_n(b))+d_\IU(f_n(b),f_n(b'))+d_\IU(f_n(b'),z')\le\\
&\le
d_\IU(z,f_n(b))+\varphi(d_A(b,b'))+d_\IU(f_n(b'),z')\le \\
&\le
d_\IU(z,f_n(b))+\varphi\big(d_A(b,a_{n+1})+d_A(a_{n+1},b')\big)+d_\IU(f_n(b'),z')\le\\
&\le
d_\IU(z,f_n(b))+\varphi(d_A(b,a_{n+1}))+\varphi(d_A(a_{n+1},b'))+d_\IU(f_n(b')',z')=
d_Y(z,y)+d_Y(y,z').
\end{aligned}
$$
On the other hand,
$$d_Y(z,y)\le d_\IU(z,f_n(b'))+\varphi(d_A(b',a_{n+1}))\le d_\IU(z,z')+d(z',f_n(b'))+\varphi(d_A(b',a_{n+1}))=d_Y(z,z')+d_Y(z',y).$$
So, the metric $d_Y$ satisfies the triangle inequality and hence is well-defined.

By \cite[Theorem 4.1]{Mel}, the identity embedding $f_n(B_n)\to\IU$ extends to an isometric embedding $e:Y\to \IU$. Let $f_{n+1}:B_{n+1}\to\IU$ be the map such that $f_{n+1}|B_n=f_n$ and $f_{n+1}(a_{n+1})=e(y)$.
Since $e(y)\notin f_n(B_n)$, the map $f_{n+1}$ is injective if so is the map $f_n$. It follows from $\w_{f_n}\le\varphi$ and $$d_\IU(f_{n+1}(x),f_{n+1}(a_{n+1}))=d_Y(f_n(x),y)\le \varphi(d_A(x,a_{n+1})) \mbox{ \ for \ }x\in B_n$$that $\w_{f_{n+1}}\le\varphi$. This completes the inductive step.
\smallskip

After the completion of the inductive construction, consider the map $f_\w:B_\w\to \IU$ defined on the set $B_\w=\bigcup_{n\in\w}B_n$ by $f_\w|B_n=f_n$ for $n\in\w$. It follows from $\w_{f_n}\le\varphi$, $n\in\w$, that $\w_{f_\w}\le\varphi$. This implies that the map $f_\w$ is uniformly continuous and hence extends to a uniformly continuous map $\bar f:A\to \IU$ having the oscillation $\w_{\bar f}\le\varphi$. It is clear that $\bar f|B=f_0=f$.
\end{proof}

Now we are able to present the {\em Proof of Theorem~\ref{main2}.}

Given a (Banach) Rakotch fractal $X$, choose a function system $\F$ consisting of (Banach) Rakotch contractions of $X$ such that $X=\bigcup_{f\in\F}f(X)$. It follows that all maps $f\in\F$ are $\varphi$-contracting for some continuity modulus $\varphi$ such that $\sup_{t\ge \delta}\varphi(t)/t<1$ for all $\delta>0$. (Moreover, if all maps $f\in\F$ are Banach contractions, then we can assume that $\sup_{t>0}\varphi(t)/t<1$).

Since the universal Urysohn space $\IU$ contains an isometric copy of each compact metric space, we can assume that $X$ is a subspace of $\IU$. By Lemma~\ref{l3.1}, each map $f\in\F$ extends to a map $\bar f:\IU\to\IU$ such that $\w_{\bar f}\le\varphi$, which implies that $\bar f$ is a (Banach) Rakotch contraction of $\IU$. By Theorems~\ref{T2} and \ref{global}, the function system $\bar \F=\{\bar f:f\in\F\}$ on the universal Urysohn space $\IU$ has an attractor $A_{\bar\F}$, which is a unique fixed point of the map $\bar\F:\IK(\IU)\to\IK(\IU)$ on the hyperspace $\IK(\IU)$ of the Urysohn space $\IU$. Taking into account that $\bar\F(X)=
\bigcup_{f\in\F}\bar f(X)=\bigcup_{f\in\F}f(X)=X$, we conclude that $X=\A_{\bar\F}$.

\end{document}